\documentclass[english]{amsart}
\usepackage[T1]{fontenc}
\usepackage[latin9]{inputenc}
\usepackage{amstext}
\usepackage{amsthm}
\usepackage{amssymb}

\makeatletter

\newcommand{\lyxmathsym}[1]{\ifmmode\begingroup\def\b@ld{bold}
  \text{\ifx\math@version\b@ld\bfseries\fi#1}\endgroup\else#1\fi}

\numberwithin{equation}{section}
\theoremstyle{plain}
\newtheorem{thm}{\protect\theoremname}[section]
\theoremstyle{definition}
\newtheorem{defn}[thm]{\protect\definitionname}
\theoremstyle{plain}
\newtheorem{lem}[thm]{\protect\lemmaname}

\makeatother

\usepackage{babel}
\providecommand{\definitionname}{Definition}
\providecommand{\lemmaname}{Lemma}
\providecommand{\theoremname}{Theorem}

\begin{document}
\title[{\footnotesize{}essential $\mathcal{F}$ sets and mixing properties} ]{essential $\mathcal{F}$ sets and mixing properties}
\author{pintu debnath and sayan goswami}
\address{{\Large{}Pintu Debnath, Department of Mathematics, Basirhat College,
Basirhat -743412, North 24th parganas, West Bengal, India.}}
\email{{\Large{}pintumath1989@gmail.com}}
\address{{\Large{}Sayan Goswami, Department of Mathematics, University of Kalyani,
Kalyani-741235, Nadia, West Bengal, India.}}
\email{{\Large{}sayan92m@gmail.com}}
\begin{abstract}
{\Large{}There is a long history of studying Ramsey theory using the
algebraic structure of Stone-\v{C}ech compactification of discrete
semigroup. It has been shown that various Ramsey theoretic structures
are contained in different algebraic large sets. In this article we
will study elementary characterization of essential $\mathcal{F}$
sets. It is known that for an $IP^{\star}$ sets $A$ in $\left(\mathbb{N},+\right)$
and sequence $\langle x_{n}\rangle_{n=1}^{\infty}$ in $\mathbb{N}$.
Then there exists sum subsystem $\langle y_{n}\rangle_{n=1}^{\infty}$
of $\langle x_{n}\rangle_{n=1}^{\infty}$ such that $FS\left(\langle y_{n}\rangle_{n=1}^{\infty}\right)\cup FP\left(\langle y_{n}\rangle_{n=1}^{\infty}\right)\subseteq A$.
In present work, We shall prove some analogous result for essential
$\mathcal{F}^{\star}$ set for some particular type of sequences.
It is well known that weak mixing( central$^{\star}$ mixing or $D^{\star}$
mixing ) implies all order weak mixing. In this article we will prove
essential $\mathcal{F}^{\star}$ mixing implies all order mixing.}{\Large\par}
\end{abstract}

\maketitle

\section{{\Large{}introduction}}

{\Large{}The Stone-\v{C}ech compactification of the set of natural
number $\mathbb{N}$ denoted by $\beta\mathbb{N}$ can be imposed
with the two operations $'+'$ and $'\cdot'$ which is an extension
of those operations on $\mathbb{N}$. The members of $\beta\mathbb{N}$
are the ultrafilters which are the subsets of the power sets of $\mathbb{N}$.
It can be shown that $\left(\beta\mathbb{N},+\right)$ and $\left(\beta\mathbb{N},\cdot\right)$
are two semigroups and they contains smallest two-sided ideals denoted
by $K\left(\beta\mathbb{N}\right)$. The ultrafilters $p\in\left(K\left(\beta\mathbb{N}\right),+\right)$
is called additively minimal ultrafilter and $p\in\left(K\left(\beta\mathbb{N}\right),\cdot\right)$
is multiplicatively minimal.}{\Large\par}

{\Large{}It can be shown that there is an one to one correspondence
between closed subsets of $\beta\mathbb{N}$ and certain collections
of subsets of $\mathbb{N}$. This certain family of subsets is called
families denoted by $\mathcal{F}$ and that closed subset will be
denoted by $\beta\mathcal{F}$. Any idempotent ultrafilter $p\in\beta\mathcal{F}$,
if exists, is called essential idempotent and any $A\in p$ is called
essential $\mathcal{F}$ set. It can be shown that if $\mathcal{F}$
is a shift invariant family, $\beta\mathcal{F}$ is an ideal of $\beta\mathbb{N}$.
An essential $\mathcal{F}^{\star}$ set is the set which belongs to
every idempotent in $\beta\mathcal{F}$. }{\Large\par}

{\Large{}There is an elementary characterization of the central sets,
$C$ sets, $D$ sets etc. which shows that those sets contains a chain
of sets with some $IP$ type property. In this article we will characterize
the essential $\mathcal{F}$-sets in terms of $\mathcal{F}$sets,
i.e. the sets belongs to the $\mathcal{F}$-family.}{\Large\par}

{\Large{}The famous Hindman's theorem says for any $A\in p$, where
$p$ is an idempotent ultrafilter contains $FS\langle x_{n}\rangle_{n=1}^{\infty}$,
the all possible finite sum of some sequence $\langle x_{n}\rangle_{n=1}^{\infty}$.
But if we partition $\mathbb{N}$ into finitely many sets, any partition
may not contain simultaneously finite sum and finite product of a
sequence. it was proved that any $IP^{\star}$, central$^{\star}$
or $C^{\star}$ sets will contain these type of configuration of some
sequences. In this article we will generalize this statement for essential
$\mathcal{F}^{\star}$ sets.}{\Large\par}

{\Large{}For a measure preserving system $\left(X,\mathcal{B},\mu,T\right)$
the study of mixing along essential $\mathcal{F}^{\star}$ sets can
be found in\cite{key-6}. Different type of mixing gives comes from
different large sets associated with the algebraic structure of $\beta\mathbb{N}$,
such as the recurrence of strong mixing comes from co finite sets,
weak mixing comes from $D^{\star}$ sets, mild mixing comes from $IP^{\star}$
sets etc. It was proved due to H.Furstenberg in\cite{key-8} that
weak mixing implies all order weak mixing. H.Furstenberg in\cite{key-8}
proved mild mixing implies all order mild mixing. In this article
we will show that any essential $\mathcal{F}^{*}$ mixing implies
all order essential $\mathcal{F}^{\star}$ mixing.}{\Large\par}

\section{{\Large{}preliminaries}}

{\Large{}Let $S$ be a discrete semigroup. In\cite{key-6}, the author
studied a detailed analysis of the closed ideals in $(\beta S,\cdot)$.
In this article we will consider $\left(S,\cdot\right)=\left(\mathbb{N},+\right)$.}{\Large\par}

{\Large{}The upward hereditary families $\mathcal{F\subseteq P}\left(\mathbb{N}\right)$
which possesses the Ramsey property and the closed subsets of $\beta\mathbb{N}$
are in one to one correspondence in nature. A collection $\mathcal{F\subseteq P}\left(\mathbb{N}\right)$
is upward hereditary if whenever $A\in\mathcal{F}$ and $A\subseteq B\subseteq\mathbb{N}$
then it follows that $B\in\mathcal{F}$. A nonempty and upward hereditary
collection $\mathcal{F\subseteq P}\left(\mathbb{N}\right)$ will be
called a family. If $\mathcal{F}$ is a family, the dual family $\mathcal{F}^{*}$
is given by,
\[
\mathcal{F}^{\star}=\{E\subseteq\mathbb{N}:\forall A\in\mathcal{F},E\cap A\neq\emptyset.\}
\]
A family $\mathcal{F}$ possesses the Ramsey property if whenever
$A\in\mathcal{F}$ and $A=A_{1}\cup A_{2}$ there is some $i\in\left\{ 1,2\right\} $
such that $A_{i}\in\mathcal{F}$.}{\Large\par}

{\Large{}There are many families $\mathcal{F}$ with Ramsay property.}{\Large\par}
\begin{itemize}
\item {\Large{}The infinite sets,}{\Large\par}
\item {\Large{}The piecewise syndetic sets,}{\Large\par}
\item {\Large{}The sets of positive upper density,}{\Large\par}
\item {\Large{}The set containing arbitrary large arithmetic progression,}{\Large\par}
\item {\Large{}The set with property that $\sum_{n\in A}\frac{1}{n}=\infty$,}{\Large\par}
\item {\Large{}The $J$ sets,}{\Large\par}
\item {\Large{}The $IP$ sets.}{\Large\par}
\end{itemize}
{\Large{}It will be easy to check that the family$\mathcal{F}$ has
the Ramsey property iff the family $\mathcal{F}^{*}$ is a filter.
For a family $\mathcal{F}$ with the Ramsey property, let $\beta(\mathcal{F})=\{p\in\beta\mathbb{N}:p\subseteq\mathcal{F}\}$.We
state and prove the following well known theorem:}{\Large\par}
\begin{thm}
{\Large{}\label{Theorem 1} For every family $\mathcal{F\subseteq P}(\mathbb{N})$
with the Ramsey property, $\beta(\mathcal{F})\subseteq\beta\mathbb{N}$
is closed and if $\mathcal{F}$ is translation invariant then $\beta(\mathcal{F})$
is left ideal.}{\Large\par}
\end{thm}

\begin{proof}
{\Large{}Let $q\in\beta\mathbb{N}\setminus\beta\left(\mathcal{F}\right)$.
Then there is $E\subset\mathbb{N}$with $E\notin\mathcal{F}$ and
$E\in q$. Now $\overline{E}$ is a neighborhood of $q$ with the
property that $\overline{E}\subset\beta\mathbb{N}\setminus\beta\left(\mathcal{F}\right)$,
This implies that $\beta\mathbb{N}\setminus\beta\left(\mathcal{F}\right)$
is open, hence $\beta\left(\mathcal{F}\right)$ is closed in $\beta\mathbb{N}$.
In order to prove that $\beta\left(\mathcal{F}\right)$ is a closed
ideal of $\beta\mathbb{N}$, it suffices to prove that $n+\beta\left(\mathcal{F}\right)\subset\beta\left(\mathcal{F}\right)$
for every $n\in\mathbb{N}$. Let $p\in\beta\left(\mathcal{F}\right)$
and $A\in n+p$ implies that $-n+A\in p\subset\mathcal{F}$. Since
$\mathcal{F}$is translation invariant, we have $A\in\mathcal{F}$.
This finishes the proof.}{\Large\par}
\end{proof}
{\Large{}Now for translation invariant $\mathcal{F\subseteq P}(\mathbb{N})$,
$\beta(\mathcal{F})$ having compact subsemigroup in $\beta\mathbb{N}$,
$\beta(\mathcal{F})$ contains atleast one idempotent. From this comment,
we get the following definition:}{\Large\par}
\begin{defn}
{\Large{}Let $\mathcal{F}$ be a translation invariant ramsay family
and $p$ be an idempotent in $\beta(\mathcal{F})$, then each member
of $p$ is called essential $\mathcal{F}$ set. And $A\subset\mathbb{N}$
is called essential $\mathcal{F}^{\star}$ set if $A$ intersects
with all essential $\mathcal{F}$ sets.}{\Large\par}
\end{defn}

{\Large{}Let us abbreviate the family of piecewise syndetic sets as
$\mathcal{PS}$, the family of positive density sets as $\Delta$
and the family of $J$ sets as $\mathcal{J}$.}{\Large\par}

{\Large{}From the above definition together with the abbreviations,
we get quasi central set is an essential $\mathcal{PS}$ set, $D$
set is an essential $\Delta$ set and $C$ set is an essential $\mathcal{J}$set.}{\Large\par}

{\Large{}Let us revisit the algebraic operation on $\beta\mathbb{N}$
in short for our purpose.}{\Large\par}

{\Large{}Identifying the principal ultrafilters with the points of
$\mathbb{N}$ and thus pretending that $\mathbb{N}\subseteq\beta\mathbb{N}$.
Given $A\subseteq\mathbb{N}$ let us set, 
\[
\overline{A}=\{p\in\beta\mathbb{N}:A\in p\}.
\]
 Then the set $\{\overline{A}:A\subseteq\mathbb{N}\}$ is a basis
for a topology on $\beta\mathbb{N}$. The operation $+$ on $\mathbb{N}$
can be extended to the Stone-\v{C}ech compactification $\beta\mathbb{N}$
of $\mathbb{N}$ so that $(\beta\mathbb{N},+)$ is a compact right
topological semigroup (meaning that for any $p\in\beta\mathbb{N}$,
the function $\rho_{p}:\beta\mathbb{N}\rightarrow\beta\mathbb{N}$
defined by $\rho_{p}(q)=q+p$ is continuous) with $\mathbb{N}$ contained
in its topological center (meaning that for any $x\in\mathbb{N}$,
the function $\lambda_{x}:\beta\mathbb{N}\rightarrow\beta\mathbb{N}$
defined by $\lambda_{x}(q)=x+q$ is continuous). Given $p,q\in\beta\mathbb{N}$
and $A\subseteq\mathbb{N}$, $A\in p+q$ if and only if $\{x\in\mathbb{N}:-x+A\in q\}\in p$,
where $-x+A=\{y\in\mathbb{N}:x+y\in A\}$. }{\Large\par}

{\Large{}A nonempty subset $I$ of a semigroup $(T,\cdot)$ is called
a left ideal of $\emph{T}$ if $T\cdot I\subset I$, a right ideal
if $I\cdot T\subset I$, and a two sided ideal (or simply an ideal)
if it is both a left and right ideal. A minimal left ideal is the
left ideal that does not contain any proper left ideal. Similarly,
we can define minimal right ideal and smallest ideal.}{\Large\par}

{\Large{}Any compact Hausdorff right topological semigroup $(T,\cdot)$
has a smallest two sided ideal}{\Large\par}

{\Large{}
\[
\begin{array}{ccc}
K(T) & = & \bigcup\{L:L\text{ is a minimal left ideal of }T\}\\
 & = & \,\,\,\,\,\bigcup\{R:R\text{ is a minimal right ideal of }T\}
\end{array}
\]
}{\Large\par}

{\Large{}Given a minimal left ideal $L$ and a minimal right ideal
$R$, $L\cap R$ is a group, and in particular contains an idempotent.
An idempotent in $K(T)$ is called a minimal idempotent. If $p$ and
$q$ are idempotents in $T$, we write $p\leq q$ if and only if $p\cdot q=q\cdot p=p$.
An idempotent is minimal with respect to this relation if and only
if it is a member of the smallest ideal. }{\Large\par}

\section{{\Large{}elementary characterization of essential $\mathcal{F}$
set}}

{\Large{}In\cite{key-6} established dynamical characterization of
essential $\mathcal{F}$ sets and elementary characterization of essential
$\mathcal{F}$ sets are still unknown. Although elementary characterization
of quasi central sets and $C$ sets are known from \cite{key-10}
and \cite{key-11} respectively. Since quasi central sets and $C$
sets are coming from by the setting of essential $\mathcal{F}$ set
and this fact confines the fact that essential $\mathcal{F}$ sets
might have elementary characterization. In this section we will prove
the supposition that elementary characterization of essential $\mathcal{F}$-sets
could be found exactly the same way what the Hindman did in \cite{key-11}
for $C$ set.}{\Large\par}

{\Large{}Let $\omega$ be the first infinite ordinal and each ordinal
indicates the set of all it's predecessor. In particular, $0=\emptyset,$
for each $n\in\mathbb{N},\:n=\left\{ 0,1,...,n-1\right\} $.}{\Large\par}
\begin{defn}
{\Large{}(a) If $f$ is a function and $dom\left(f\right)=n\in\omega$,
then for all $x$, $f^{\frown}x=f\cup\left\{ \left(n,x\right)\right\} $.}{\Large\par}

{\Large{}(b) Let $T$ be a set functions whose domains are members
of $\omega$. For each $f\in T$, $\mathcal{B}_{f}\left(T\right)=\left\{ x:f^{\frown}x\in T\right\} .$}{\Large\par}
\end{defn}

\begin{lem}
{\Large{}Let $p\in\beta\mathbb{N}$. Then $p$ is an idempotent if
and only if for each $A\in p$ there is a non-empty set $T$ of functions
such that }{\Large\par}
\end{lem}

\begin{enumerate}
\item {\Large{}For all $f\in T$, $dom\left(f\right)\in\omega$ and $range\left(f\right)\subseteq A$.}{\Large\par}
\item {\Large{}For all $f\in T$, $\mathcal{B}_{f}\left(T\right)\in p$.}{\Large\par}
\item {\Large{}For all $f\in T$ and any $x\in\mathcal{B}_{f}\left(T\right)$,
$\mathcal{B}_{f^{\frown}x}\left(T\right)\subseteq x^{-1}\mathcal{B}_{f}\left(T\right)$.}{\Large\par}
\end{enumerate}
\begin{thm}
{\Large{}Let $A\subseteq\mathbb{N}$ then, all the statements are
equivalent.}{\Large\par}

{\Large{}(a) $A$ is an essential $\mathcal{F}$ set.}{\Large\par}

{\Large{}(b)There is a non empty set $T$ of function such that:}{\Large\par}

{\Large{}(i) For all $f\in T$,$\text{domain}\left(f\right)\in\omega$
and $rang\left(f\right)\subseteq A$.}{\Large\par}

{\Large{}(ii) For all $f\in T$ and all $x\in B_{f}\left(T\right)$,
$B_{f^{\frown}x}\subseteq-x+B_{f}\left(T\right)$}{\Large\par}

{\Large{}(iii) For all $F\in\mathcal{P}_{f}\left(T\right)$, $\cap_{f\in F}B_{f}(T)$
is a $\mathcal{F}$ set.}{\Large\par}

{\Large{}There is an $FP$ tree $T$ in $A$ such that for each $F\in\mathcal{P}_{f}\left(T\right),\,\bigcap_{f\in F}B_{f}$
is a $\mathcal{F}$ set.}{\Large\par}

{\Large{}(c) There is a downward directed family $\left\langle C_{F}\right\rangle _{F\in I}$
of subsets of $A$ such that}{\Large\par}

{\Large{}(i) for each $F\in I$ and each $x\in C_{F}$ there exists
$G\in I$ with $C_{G}\subseteq x^{-1}C_{F}$ and}{\Large\par}

{\Large{}(ii) for each $\mathcal{F}\in\mathcal{P}_{f}\left(I\right),\,\bigcap_{F\in\mathcal{F}}C_{F}$
is a $\mathcal{F}$-set. }{\Large\par}

{\Large{}(d) There is a decreasing sequence $\left\langle C_{n}\right\rangle _{n=1}^{\infty}$
of subsets of $A$ such that }{\Large\par}

{\Large{}(i) for each $n\in\mathbb{N}$ and each $x\in C_{n}$, there
exists $m\in\mathbb{N}$ with $C_{m}\subseteq x^{-1}C_{n}$ and }{\Large\par}

{\Large{}(ii) for each $n\in\mathbb{N}$, $C_{n}$ is a $\mathcal{F}$
set.}{\Large\par}

\end{thm}

\begin{proof}
{\Large{}(a) $\Rightarrow$ (b). As $A$ be an essential $\mathcal{F}$
set, then there exists an idempotent $p\in\beta\left(\mathcal{F}\right)$
such that $A\in p$. Pick a set $T$ of functions as guaranteed by
Lemma 2.3 conclusions $(i)$ and $(ii)$ hold directly. Given $F\in\mathcal{P}_{f}\left(T\right)$,
$B_{f}\in p$ for all $f\in F$, hence $\bigcap_{f\in F}B_{f}\in p$
so $\bigcap_{f\in F}B_{f}$ is a $\mathcal{F}$-set.}{\Large\par}

{\Large{}(b) $\Rightarrow$ (c). Let $T$ be guaranteed by Let $I=\mathcal{P}_{f}\left(T\right)$
and for each $F\in I$, let $C_{F}=\bigcap_{f\in F}B_{f}$. Then directly
each $C_{F}$ is a $\mathcal{F}$ -set. Given $\mathcal{F}\in\mathcal{P}_{f}\left(I\right)$,
if $G=\bigcup\mathcal{F}$, then $\bigcap_{F\in\mathcal{F}}C_{F}=C_{G}$
and is therefore a $\mathcal{F}$-set. To verify (i), let $F\in I$
and let $x\in C_{F}$. Let $G=\left\{ f^{\frown}x:f\in F\right\} $.
For each $f\in F$, $B_{f^{\frown}x}\subseteq-x+B_{f}$ and so $C_{G}\subseteq-x+C_{F}$.}{\Large\par}

{\Large{}(c) $\Rightarrow$ (a). Let $\langle C_{F}\rangle$ is guaranteed
by (c). Let $M=\bigcap_{F\in I}\overline{C_{F}}$. By \cite[Theorem 4.20]{key-14},
$M$ is a subsemigroup of $\beta\mathbb{N}$. By \cite[Theorem 3.11]{key-14}
there is some $p\in\beta\mathbb{N}$ such that $\left\{ C_{F}:F\in I\right\} \subseteq p\subseteq\mathcal{F}$.
Therefore $M\cap\beta\left(\mathcal{F}\right)\neq\emptyset$; and
so $M\cap\beta\left(\mathcal{F}\right)$ is a compact subsemigroup
of $\beta\mathbb{N}$. Thus there is an idempotent $p\in M\cap\beta\left(\mathcal{F}\right)$
and so , $A$ is an essential $\mathcal{F}$-set. }{\Large\par}

{\Large{}It is trivial that (d) implies (c). Assume now that $S$
is countable. We shall show that (b) implies (d). So let $T$ be as
guaranteed by (b). Then $T$ is countable so enumerate $T$ as $\left\{ f_{n}:n\in\mathbb{N}\right\} $.
For $n\in\mathbb{N}$, let $C_{n}=\bigcap_{k=1}^{n}B_{f_{k}}$. Then
each $C_{n}$ is a $\mathcal{F}$set. Let $n\in\mathbb{N}$ and let
$x\in C_{n}$. Pick $m\in\mathbb{N}$ such that 
\[
\left\{ f_{k}^{\frown}x:k\in\left\{ 1,2,\ldots,n\right\} \right\} \subseteq\left\{ f_{1},f_{2},\ldots,f_{m}\right\} .
\]
 Then $C_{m}\subseteq-x+C_{n}$.}{\Large\par}
\end{proof}

\section{{\Large{}combined additive and multiplicative structure}}

{\Large{}In this section and in the next section, we consider $\mathcal{F}$
is dilation invariant with translation invariant.}{\Large\par}

{\Large{}Given a sequence $\langle x_{n}\rangle_{n=1}^{\infty}$ in
$\mathbb{N}$, we say that $\langle y_{n}\rangle_{n=1}^{\infty}$
is a sum subsystem of $\langle x_{n}\rangle_{n=1}^{\infty}$ provided
there exists a sequence $\langle H_{n}\rangle_{n=1}^{\infty}$ of
non-empty finite subset such that $\text{max}H_{n}<\text{min}H_{n+1}$
and $y_{n}=\sum_{t\in H_{n}}x_{t}$ for each $n\in\mathbb{N}$. In
\cite{key-5} Hindman and Bergelson characterized $IP^{*}$ sets by
introducing the following theorem.}{\Large\par}
\begin{thm}
{\Large{}Let $\langle x_{n}\rangle_{n=1}^{\infty}$ be a sequence
in $\mathbb{N}$ and $A$ be $IP^{*}$ set in $\left(\mathbb{N},+\right)$.
Then there exists a subsystem $\langle y_{n}\rangle_{n=1}^{\infty}$
of $\langle x_{n}\rangle_{n=1}^{\infty}$ such that}\\
{\Large{} $FS\left(\langle y_{n}\rangle_{n=1}^{\infty}\right)\cup FP\left(\langle y_{n}\rangle_{n=1}^{\infty}\right)\subseteq A$.}{\Large\par}
\end{thm}

{\Large{}In \cite{key-8}, D. De showed that central$^{\star}$ sets
also possess some $IP^{\star}$ set-like properties for some specified
sequences called minimal sequence:}{\Large\par}
\begin{defn}
{\Large{}A sequence $\langle x_{n}\rangle_{n=1}^{\infty}$ in $\mathbb{N}$
is minimal sequence if 
\[
\cap_{m=1}^{\infty}cl\left(FS\left(\langle x_{n}\rangle_{n=m}^{\infty}\right)\right)\cap K\left(\beta\mathbb{N}\right)\neq\emptyset.
\]
}{\Large\par}
\end{defn}

{\Large{}It is known that $\langle2^{n}\rangle_{n=1}^{\infty}$ is
a minimal sequence while the sequence $\langle2^{2n}\rangle_{n=1}^{\infty}$
is not a minimal sequence. In \cite{key-1} it is proved that in $\left(\mathbb{N},+\right)$
minimal sequences are nothing but those for which $FS\left(\langle x_{n}\rangle_{n=1}^{\infty}\right)$
is picewise syndetic i.e. $cl\left(FS\left(\langle x_{n}\rangle_{n=1}^{\infty}\right)\right)\cap K\left(\beta\mathbb{N}\right)\neq\emptyset$.
And in \cite{key-8} D.De proved the following substantial multiplicative
result of central$^{\star}$ sets:}{\Large\par}
\begin{thm}
{\Large{}Let $\langle x_{n}\rangle_{n=1}^{\infty}$ be a minimal sequence
in $\mathbb{N}$ and $A$ be central$^{*}$ set in $\left(\mathbb{N},+\right)$.
Then there exists a subsystem $\langle y_{n}\rangle_{n=1}^{\infty}$
of $\langle x_{n}\rangle_{n=1}^{\infty}$ such that $FS\left(\langle y_{n}\rangle_{n=1}^{\infty}\right)\cup FP\left(\langle y_{n}\rangle_{n=1}^{\infty}\right)\subseteq A$.}{\Large\par}
\end{thm}

{\Large{}In \cite{key-7} D. De established an analog version of the
above theorem in case of $C^{*}$ sets for some specific type of sequences
called almost minimal sequence.}{\Large\par}
\begin{defn}
{\Large{}A sequence $\langle x_{n}\rangle_{n=1}^{\infty}$ in $\mathbb{N}$
is almost minimal sequence if 
\[
\cap_{m=1}^{\infty}cl\left(FS\left(\langle x_{n}\rangle_{n=m}^{\infty}\right)\right)\cap J\left(\mathbb{N}\right)\neq\emptyset.
\]
}{\Large\par}
\end{defn}

{\Large{}In\cite{key-7}, with help of N. Hindman, D. De introduced
an example of almost minimal sequence which is not minimal sequence
and characterized the almost minimal sequences by the following theorem:}{\Large\par}
\begin{thm}
{\Large{}In $(\mathbb{N},+)$ the following conditions are equivalent:}{\Large\par}

{\Large{}(1) $\langle x_{n}\rangle_{n=1}^{\infty}$ is almost minimal
sequence.}{\Large\par}

{\Large{}(2) $FS(\langle x_{n}\rangle_{n=1}^{\infty})$ is a $J$
set.}{\Large\par}

{\Large{}(3) There is an idempotent in $\cap_{m=1}^{\infty}cl\left(FS\left(\langle x_{n}\rangle_{n=m}^{\infty}\right)\right)\cap J\left(\mathbb{N}\right)$.}{\Large\par}
\end{thm}

{\Large{}Now it is write place to state the main theorem of \cite{key-7}:}{\Large\par}
\begin{thm}
{\Large{}Let $\langle x_{n}\rangle_{n=1}^{\infty}$ be a minimal sequence
in $\mathbb{N}$ and $A$ be $C^{*}$ set in $\left(\mathbb{N},+\right)$.
Then there exists a subsystem $\langle y_{n}\rangle_{n=1}^{\infty}$
of $\langle x_{n}\rangle_{n=1}^{\infty}$ such that 
\[
FS\left(\langle y_{n}\rangle_{n=1}^{\infty}\right)\cup FP\left(\langle y_{n}\rangle_{n=1}^{\infty}\right)\subseteq A.
\]
}{\Large\par}
\end{thm}

{\Large{}As we know that $C$ sets are essential $\mathcal{J}$ sets,
the above theorem motives us to think some analog result for essential
$\mathcal{F}$ sets. First let us define almost $\mathcal{F}$ minimal
sequence.}{\Large\par}
\begin{defn}
{\Large{}A sequence $\langle x_{n}\rangle_{n=1}^{\infty}$ in $\mathbb{N}$
is almost $\mathcal{F}$ minimal sequence if 
\[
\cap_{m=1}^{\infty}cl\left(FS\left(\langle x_{n}\rangle_{n=m}^{\infty}\right)\right)\cap\beta\left(F\right)\neq\emptyset
\]
}{\Large\par}
\end{defn}

{\Large{}We can characterize almost $\mathcal{F}$ minimal sequences
as like as almost minimal sequence given below and can be proved in
the same way as the author did in \cite{key-7} for almost minimal
sequences:}{\Large\par}
\begin{thm}
{\Large{}In $(\mathbb{N},+)$ the following conditions are equivalent:}{\Large\par}

{\Large{}(a) $\langle x_{n}\rangle_{n=1}^{\infty}$ is almost $\mathcal{F}$
minimal sequence.}{\Large\par}

{\Large{}(b) $FS(\langle x_{n}\rangle_{n=1}^{\infty})\in q$, for
some $q\in\beta(\mathcal{F})$.}{\Large\par}

{\Large{}(c) There is an idempotent in $\cap_{m=1}^{\infty}cl(FS(\langle x_{n}\rangle_{n=m}^{\infty}))\cap\beta(\mathcal{F})$.}{\Large\par}
\end{thm}

\begin{proof}
{\Large{}$(a)\implies(b)$ follows from definition.}{\Large\par}

{\Large{}$(b)\implies(c):$ Since $FS\left(\langle x_{n}\rangle_{n=1}^{\infty}\right)\in q\in\beta(\mathcal{F})$
we get}\\
{\Large{} $cl\left(FS\left(\langle x_{n}\rangle_{n=1}^{\infty}\right)\right)\cap\beta\left(\mathcal{F}\right)\neq\emptyset$.
By \cite[Lemma 5.11]{key-14} Choose $\cap_{m=1}^{\infty}cl\left(FS\left(\langle x_{n}\rangle_{n=m}^{\infty}\right)\right)$.
It will easy to see that $\cap_{m=1}^{\infty}cl\left(FS\left(\langle x_{n}\rangle_{n=1}^{\infty}\right)\right)$
is a closed subsemigroup of $\beta\mathbb{N}$ and as well as $\beta(\mathcal{F})$
is also closed subsemigroup $\beta\mathbb{N}$. Hence $\,\cap_{m=1}^{\infty}cl\left(FS\left(\langle x_{n}\rangle_{n=m}^{\infty}\right)\right)\cap\beta\left(\mathcal{F}\right)$
is a compact subsemigroup of $\left(\beta\mathbb{N},+\right)$. So
it will be sufficient to check that $\cap_{m=1}^{\infty}cl\left(FS\left(\langle x_{n}\rangle_{n=m}^{\infty}\right)\right)\cap\beta\left(\mathcal{F}\right)\neq\emptyset$.}{\Large\par}

{\Large{}Now choose arbitrarily $m\in\mathbb{N}$ and then}\\
{\Large{} $FS\left(\langle x_{n}\rangle_{n=1}^{\infty}\right)=FS\left(\langle x_{n}\rangle_{n=m}^{\infty}\right)\cup FS\left(\langle x_{n}\rangle_{n=1}^{m-1}\right)\cup\left\{ t+FS\left(\langle x_{n}\rangle_{n=m}^{\infty}\right):t\in FS\left(\langle x_{n}\rangle_{n=1}^{m-1}\right)\right\} $
and so we have one of the followings:}{\Large\par}

{\Large{}$1.\,FS\left(\langle x_{n}\rangle_{n=m}^{\infty}\right)\in p$}{\Large\par}

{\Large{}$2.\,FS\left(\langle x_{n}\rangle_{n=1}^{m-1}\right)\in p$}{\Large\par}

{\Large{}$3.\,t+FS\left(\langle x_{n}\rangle_{n=m}^{\infty}\right)\in p$
for some $t\in FS\left(\langle x_{n}\rangle_{n=1}^{m-1}\right)$}{\Large\par}

{\Large{}Now (1) is not possible as in that case $p$ will be a member
of principle ultrafilter. If (2) holds then we have done. Now if we
assume (3) holds then for some}\\
{\Large{} $t\in FS\left(\langle x_{n}\rangle_{n=1}^{m-1}\right)$,
we have $t+FS\left(\langle x_{n}\rangle_{n=m}^{\infty}\right)\in p$.
Choose $q\in cl\left(FS\left(\langle x_{n}\rangle_{n=m}^{\infty}\right)\right)$
so that $t+q=p$. Now for every $F\in q$, $t\in\left\{ n\in\mathbb{N}:-n+\left(t+F\right)\in q\right\} $
so that $t+F\in p$. Since $\mathcal{F-}$sets are translation invariant,
$F$ is a $\mathcal{F-}$sets. We have $q\in\beta\left(\mathcal{F}\right)\cap cl\left(FS\left(\langle x_{n}\rangle_{n=m}^{\infty}\right)\right)$.}{\Large\par}

{\Large{}$(c)\implies(a)$ follows from definition of $\mathcal{F}$
minimal sequence and condition (3).}{\Large\par}
\end{proof}
{\Large{}To prove the main theorem of this section following to lemmas
are essential.}{\Large\par}
\begin{lem}
{\Large{}If $A$ be an essential $\mathcal{F}$ set in $(\mathbb{N},+)$
then $nA$ is also an essential $\mathcal{F}$ set in $(\mathbb{N},+)$
for any $n\in\mathbb{N}$.}{\Large\par}
\end{lem}

\begin{proof}
{\Large{}If $A$ be an essential $\mathcal{F}$ set, then by elementary
characterization of essential $\mathcal{F}$ set we get a sequence
of $\mathcal{F}$ sets $\left\langle C_{k}\right\rangle _{k=1}^{\infty}$
with $A\supseteq C_{1}\supseteq C_{2}\supseteq\cdots$ such that for
each $k\in\mathbb{N}$ and each $t\in C_{k}$, there exists $p\in\mathbb{N}$
with $C_{p}\subseteq-t+C_{k}$. Now consider the sequence $\left\langle nC_{k}\right\rangle _{k=1}^{\infty}$
of $\mathcal{F}$ sets which satisfies $nA\supseteq nC_{1}\supseteq nC_{2}\supseteq\cdots$
and for each $k\in\mathbb{N}$ and each $t\in nC_{k}$, there exists
$p\in\mathbb{N}$ with $nC_{p}\subseteq-t+nC_{k}$. This proves that
$nA$ is an essential $\mathcal{F}$ set in $(\mathbb{N},+)$ for
any $n\in\mathbb{N}$.}{\Large\par}
\end{proof}
{\Large{}We get another lemma given below.}{\Large\par}
\begin{lem}
{\Large{}If $A$ be an essential $\mathcal{F}^{\star}$ set in $(\mathbb{N},+)$
then $n^{-1}A$ is also a essential $\mathcal{F^{\star}}$ set in
$(\mathbb{N},+)$ for any $n\in\mathbb{N}$.}{\Large\par}
\end{lem}

\begin{proof}
{\Large{}It is sufficient to show that for any essential $\mathcal{F}$
set $B$, $B\cap n^{-1}A\neq\emptyset$. Since $B$ is essential $\mathcal{F}$
set , $nB$ is essential $\mathcal{F}$ set and $A\cap nB\neq\emptyset$.
Choose $m\in A\cap nB$ and $k\in B$ such that $m=nk$. Therefore
$m=nk$. Therefore $k=m/n\in n^{-1}A$ so $B\cap n^{-1}A\neq\emptyset$.}{\Large\par}
\end{proof}
{\Large{}We now show that all $\mathcal{F}^{*}$ set have a substantial
multiplicative property.}{\Large\par}
\begin{thm}
{\Large{}Let $\langle x_{n}\rangle_{n=1}^{\infty}$ be a $\mathcal{F}$
minimal sequence and $A$ be a in $(\mathbb{N},+)$. Then there exists
a sum subsystem $\langle y_{n}\rangle_{n=1}^{\infty}$ of $\langle x_{n}\rangle_{n=1}^{\infty}$
such that}\\
{\Large{} $FS(\langle y_{n}\rangle_{n=1}^{\infty})\cup FP(\langle y_{n}\rangle_{n=1}^{\infty})\subseteq A$.}{\Large\par}
\end{thm}

\begin{proof}
{\Large{}since $FS(\langle x_{n}\rangle_{n=1}^{\infty})$ is almost
$\mathcal{F}$ minimal sequence in $\mathbb{N}$. We can find some
essential idempotent $p\in\beta\left(\mathcal{F}\right)$ for when
$FS(\langle x_{n}\rangle_{n=1}^{\infty})\in p$. Since A be a $\mathcal{F}^{*}$
set for every $n\in\mathbb{N}$, $n^{-1}A\in p$. Let $A^{*}=\{n\in A:-n+A\in p\}$,
then $A^{*}\in p$. We can choose $y_{1}\in A^{*}\cap FS(\langle x_{n}\rangle_{n=1}^{\infty})$.
Inductively, let $m\in\mathbb{N}$ and $\langle y_{i}\rangle_{i=1}^{m}$,
$\langle H_{i}\rangle_{i=1}^{m}$ in $\mathcal{P}_{f}(\mathbb{N})$
be chosen with the following property:}{\Large\par}

{\Large{}$1.\,i\in\{1,2,\ldots\,,m-1\}\,\text{max}H_{i}<\text{min}H_{i+1}$}{\Large\par}

{\Large{}$2.\,\text{If}\,y_{i}={\displaystyle \sum_{t\in H_{i}}x_{t}\,\text{then}\,\sum_{t\in H_{i}}x_{t}\in A^{*}\,\text{and}\,FS(\langle y_{i}\rangle_{i=1}^{m})\subseteq A}.$}{\Large\par}

{\Large{}We observe that $\{\,\sum_{t\in H}x_{t}:\,H\in\mathcal{P}_{f}(\mathbb{N}),\,\text{min}H>\text{max}H_{m}\}$.
It follows that we can choose $H_{m+1}\in\mathcal{P}_{f}(\mathbb{N})$
such that $\text{min}H_{m+1}>\text{max}H_{m},\,\sum_{t\in H_{m+1}}x_{t}\in A^{*}$,
$\sum_{t\in H_{m+1}}x_{t}\in-n+A^{*}$ for every $n\in FS(\langle y_{i}\rangle_{i=1}^{m})$
and $\sum_{t\in H_{m+1}}x_{t}\in n^{-1}A^{*}$for every $n\in FS(\langle y_{i}\rangle_{i=1}^{m})$.
Putting $y_{m+1}=\sum_{t\in H_{m+1}}x_{t}$. Show the induction can
be continued and proves the theorem.}{\Large\par}
\end{proof}

\section{{\Large{}mixing implies all order mixing}}

{\Large{}By a measure preserving dynamical system(MDS) we mean set
$(X,\mathcal{B},\mu,T)$, where $(X,\mathcal{B},\mu)$ is a Lebesgue
space and $T$ is invertable and measure preserving space from compact
matric space $X$ to $X$.We say a measurable function $f\in L^{2}\left(X\right)$
is rigid for $T$ if for some sequence $\langle n_{k}\rangle_{k=1}^{\infty}$
such that $T^{n_{k}}f\to f$ in $L^{2}\left(X\right)$. Before starting
the disscussion about this section, let us define strong mixing, mild
mixing and weak mixing.}{\Large\par}
\begin{enumerate}
\item {\Large{}A measure preserving dynamical system $(X,\mathcal{B},\mu,T)$
is weak mixing if for all $A,B\in\mathcal{B}$ with $\mu\left(A\right)\mu\left(B\right)>0$
\[
\text{lim}_{n\to\infty}\frac{1}{n}\sum_{i=0}^{n-1}\mid\mu\left(A\cap T^{-n}B\right)-\mu\left(A\right)\mu\left(B\right)\mid=0
\]
}{\Large\par}
\item {\Large{}A measure preserving dynamical system $(X,\mathcal{B},\mu,T)$
is strong mixing if for all $A,B\in\mathcal{B}$ with $\mu\left(A\right)\mu\left(B\right)>0$
\[
\text{lim}_{n\to\infty}\mu\left(A\cap T^{-n}B\right)=\mu\left(A\right)\mu\left(B\right)
\]
}{\Large\par}
\item {\Large{}A measure preserving dynamical system $(X,\mathcal{B},\mu,T)$
is mild mixing if there are no non-constant rigid function $f\in L^{2}\left(X\right)$.}{\Large\par}
\end{enumerate}
{\Large{}It is it is interesting and motivational result that mild
mixing and weak mixing can be characterized in terms of $IP^{*}$sets,
central$^{\star}$sets and $D^{\star}$ sets. We know the following
theorem from \cite{key-10} proved by H. Frustenberg, connects weak
mixing with $D^{\star}$set.}{\Large\par}
\begin{thm}
{\Large{}The measure preserving dynamical system$(X,B,\mu,T)$ is
Weak mixing iff for any $A,B\in\mathcal{B}$ with $\mu\left(A\right)\mu\left(B\right)>0$
and any $\varepsilon>0$, the set 
\[
\left\{ n\in\mathbb{N}:\mid\mu\left(A\cap T^{-n}B\right)-\mu\left(A\right)\mu\left(B\right)\mid<\varepsilon\right\} 
\]
is a $D^{\star}$ set.}{\Large\par}
\end{thm}

{\Large{}Another same potentially significant equivalent condition
was proved by V. Bergelson in \cite{key-3}connects weak mixing with
central$^{\star}$set.}{\Large\par}
\begin{thm}
{\Large{}The measure preserving dynamical system$(X,B,\mu,T)$ is
Weak mixing iff for any $A,B\in\mathcal{B}$ with $\mu\left(A\right)\mu\left(B\right)>0$
and any $\varepsilon>0$, the set 
\[
\left\{ n\in\mathbb{N}:\mid\mu\left(A\cap T^{-n}B\right)-\mu\left(A\right)\mu\left(B\right)\mid<\varepsilon\right\} 
\]
is a central$^{\star}$ set.}{\Large\par}
\end{thm}

{\Large{}Mild mixing is connected with $IP^{\star}$ set was proved
by H. Frustenberg in \cite{key-10} as follow:}{\Large\par}
\begin{thm}
{\Large{}The measure preserving dynamical system$(X,B,\mu,T)$ is
Weak mixing iff for any $A,B\in\mathcal{B}$ with $\mu\left(A\right)\mu\left(B\right)>0$
and any $\varepsilon>0$, the set 
\[
\left\{ n\in\mathbb{N}:\mid\mu\left(A\cap T^{-n}B\right)-\mu\left(A\right)\mu\left(B\right)\mid<\varepsilon\right\} 
\]
 is an $IP^{\star}$ set.}{\Large\par}
\end{thm}

{\Large{}Since $D$ sets are essential $\mathcal{\Delta}$ set, and
we connect a mixing with $D$ set, naturally a question arises whether
we can associate a specific mixing with a specific essintial $\mathcal{F}$-set.
In \cite{key-6} C. Chistoferson first introduced essential $\mathcal{F}^{\star}$mixing
and proved some essential results of $\mathcal{F}^{\star}$-mixing.}{\Large\par}
\begin{thm}
{\Large{}The measure preserving dynamical system$(X,B,\mu,T)$ is
essential $\mathcal{F}^{\star}$mixing if for any $A,B\in\mathcal{B}$(or
$f,g\in L^{2}\left(X\right)$) with $\mu\left(A\right)\mu\left(B\right)>0$
and any $\varepsilon>0$, the set 
\[
\begin{array}{c}
\left\{ n\in\mathbb{N}:\mid\mu\left(A\cap T^{-n}B\right)-\mu\left(A\right)\mu\left(B\right)\mid<\varepsilon\right\} \\
(or\,\left\{ n\in\mathbb{N}:\mid\int_{X}f\left(x\right)g\left(T^{n}x\right)d\mu\left(x\right)-\int_{X}f\left(x\right)d\mu\left(x\right)\int_{X}g\left(x\right)d\mu\left(x\right)\mid<\varepsilon\right\} )
\end{array}
\]
 is an essential $\mathcal{F}^{\star}$ set.}{\Large\par}
\end{thm}

{\Large{}Now, We know from \cite{key-10}, weak mixing implies all
order weak mixing.}{\Large\par}
\begin{thm}
{\Large{}The measure preserving dynamical system$(X,B,\mu,T)$ is
Weak mixing iff for any $A_{0},A_{1},\ldots\,,A_{k}\in\mathcal{B}$
with}\\
{\Large{} $\mu\left(A_{0}\right)\mu\left(A_{1}\right)\cdots\mu\left(A_{k}\right)>0$
and $n_{1},\ldots,n_{k}\in\mathbb{N}$ with $n_{1}<\ldots<n_{k}$
any $\varepsilon>0$, the set 
\[
\left\{ n\in\mathbb{N}:\mid\mu\left(A_{0}\cap T^{-nn_{1}}A_{1}\cap\ldots\cap T^{-nn_{k}}A_{k}\right)-\mu\left(A_{0}\right)\mu\left(A_{1}\right)\cdots\mu\left(A_{k}\right)\mid<\varepsilon\right\} 
\]
is a $D^{\star}$ set.}{\Large\par}
\end{thm}

{\Large{}Following is an analog version of mild mixing known from
\cite{key-10}:}{\Large\par}
\begin{thm}
{\Large{}The measure preserving dynamical system$(X,B,\mu,T)$ is
Mild mixing iff for any $A_{0},A_{1},\ldots\,,A_{k}\in\mathcal{B}$
with}\\
{\Large{} $\mu\left(A_{0}\right)\mu\left(A_{1}\right)\cdots\mu\left(A_{k}\right)>0$
and $n_{1},\ldots,n_{k}\in\mathbb{N}$ with $n_{1}<\ldots<n_{k}$
any $\varepsilon>0$, the set 
\[
\left\{ n\in\mathbb{N}:\mid\mu\left(A_{0}\cap T^{-nn_{1}}A_{1}\cap\ldots\cap T^{-nn_{k}}A_{k}\right)-\mu\left(A_{0}\right)\mu\left(A_{1}\right)\cdots\mu\left(A_{k}\right)\mid<\varepsilon\right\} 
\]
is an $IP^{\star}$ set.}{\Large\par}
\end{thm}

{\Large{}From the Theorem 5.5, a question arises , whether essential
$\mathcal{F}^{\star}$-mixing set implies all order essential $\mathcal{F}^{\star}$-mixing.
The main result of this section is an affirmative answer of this question.
The following lemma was proved in \cite{key-15} by C. Schnell in
a sophisticated manner and posted in the blog of J. Moreira, is main
ingredient of our main result.}{\Large\par}
\begin{lem}
{\Large{}Let $p$ be an idempotent ultrafilter and let $\{x_{n}\}_{n}$
be a bounded sequence in $H$ (Hilbert Space) such that for each $d\in\mathbb{N}$
we have $p\lyxmathsym{\textendash}{\displaystyle \lim_{n}\langle x_{n+d},x_{n}\rangle=0}$.
Then also $p\lyxmathsym{\textendash}{\displaystyle \lim_{n}x_{n}=0}$
weakly.}{\Large\par}
\end{lem}

\begin{proof}
{\Large{}For each $N\in\mathbb{N}$ we have that,}{\Large\par}

{\Large{}${\displaystyle p-\lim_{n}x_{n}=p-\lim_{n_{1}}p-\lim_{n_{2}}\cdots p-\lim_{n_{N}}\frac{1}{N}\sum_{k=1}^{N}x_{n_{k}+...+n_{N}}}$}{\Large\par}

{\Large{}Taking norms and using the Cauchy-Schwartz inequality}{\Large\par}

{\Large{}
\begin{align*}
{\displaystyle \left\Vert p-\lim_{n\to\infty}x_{n}\right\Vert ^{2}} & \leq{\displaystyle p-\lim_{n_{1}}p-\lim_{n_{2}}\cdots p-\lim_{n_{N}}\left\Vert \sum_{k=1}^{N}x_{n_{k}+\cdots+n_{N}}\right\Vert ^{2}}\\
 & ={\displaystyle \frac{1}{N^{2}}\sum_{k,l=1}^{N}p-\lim_{n_{1}}p-\lim_{n_{2}}\cdots p-\lim_{n_{N}}\left\langle x_{n_{k}+\cdots+n_{N}},x_{n_{l}+\cdots+n_{N}}\right\rangle }\\
 & ={\displaystyle p-\lim_{n_{1}}p-\lim_{n_{2}}\cdots p-\lim_{n_{N}}\frac{1}{N^{2}}\left\langle \sum_{k=1}^{N}x_{n_{k}+\cdots+n_{N}},\sum_{l=1}^{N}x_{n_{l}+\cdots+n_{N}}\right\rangle }\\
 & ={\displaystyle p-\lim_{n_{1}}p-\lim_{n_{2}}\cdots p-\lim_{n_{N}}\frac{1}{N^{2}}\left\langle \sum_{k=1}^{N}x_{n_{k}+\cdots+n_{N}},\sum_{l=1}^{N}x_{n_{l}+\cdots+n_{N}}\right\rangle }\\
 & ={\displaystyle \frac{1}{N^{2}}\sum_{k=1}^{N}p-\lim_{n_{k}}\|x_{n_{k}}\|^{2}+\frac{2}{N^{2}}\sum_{k<l}p-\lim_{n_{k}}p-\lim_{n_{l}}\langle x_{n_{k}+n_{l}},x_{n_{l}}\rangle}\\
 & ={\displaystyle \frac{1}{N}p-\lim_{n}\|x_{n}\|^{2}}
\end{align*}
}{\Large\par}

{\Large{}Since $\left\{ N\right\} $ was chosen arbitrarily we conclude
that ${\displaystyle \left\Vert p-\lim_{n}x_{n}\right\Vert ^{2}=0}.$}{\Large\par}
\end{proof}
{\Large{}Also the following simple lemma is same important to prove
our main theorem of this section and follows from defination of essential
$\mathcal{F}^{\star}$ mixing and lemma 4.10.}{\Large\par}
\begin{lem}
{\Large{}The measure preserving dynamical system $\left(X,\mathcal{B},\mu,T\right)$
is essential $\mathcal{F}^{\star}$ mixing iff $\left(X,\mathcal{B},\mu,T^{n}\right)$
is also essential $\mathcal{F}^{\star}$-mixing for all $n\in\mathbb{N}$.}{\Large\par}
\end{lem}

{\Large{}Now, we are in the right position of proving the main theorem
of this section. The technique of this proof is traditional.}{\Large\par}
\begin{thm}
{\Large{}The measure preserving dynamical system $(X,B,\mu,T)$ is
$\mathcal{F^{\star}}$ mixing iff for any $A_{0},A_{1},\ldots\,,A_{k}\in\mathcal{B}$
with $\mu\left(A_{0}\right)\mu\left(A_{1}\right)\cdots\mu\left(A_{k}\right)>0$
and $n_{1},\ldots,n_{k}\in\mathbb{N}$ with $n_{1}<\ldots<n_{k}$
any $\varepsilon>0$, the set 
\[
\left\{ n\in\mathbb{N}:\mid\mu\left(A_{0}\cap T^{-nn_{1}}A_{1}\cap\ldots\cap T^{-nn_{k}}A_{k}\right)-\mu\left(A_{0}\right)\mu\left(A_{1}\right)\cdots\mu\left(A_{k}\right)\mid<\varepsilon\right\} 
\]
is an $\mathcal{F^{\star}}$set.}{\Large\par}
\end{thm}

\begin{proof}
{\Large{}In this theorem we prove for $k=2$.}{\Large\par}

{\Large{}It remains to show $p-{\displaystyle \lim_{n}}\mu\left(A_{0}\cap T^{-nn_{1}}A_{1}\cap T^{-nn_{2}}A_{2}\right)=\mu\left(A_{0}\right)\mu\left(A_{1}\right)\mu\left(A_{2}\right)$}{\Large\par}

{\Large{}Let $a_{n}\left(x\right)=1_{A_{1}}\left(T^{nn_{1}}x\right)1_{A_{2}}\left(T^{nn_{2}}x\right)-\mu\left(A_{1}\right)\mu\left(A_{2}\right)$. }{\Large\par}

{\Large{}We will show $p-lim_{n}a_{n}$ with respect to the weak topology.}{\Large\par}

{\Large{}Since $T$ is strongly mixing we have }{\Large\par}

{\Large{}$p{\displaystyle -}\left(p-{\displaystyle \lim_{n}}\left(a_{n+m},a_{n}\right)\right)$}{\Large\par}

{\Large{}$=p-{\displaystyle \lim_{m}}\left(p-{\displaystyle \lim_{n}}\int1_{A_{1}}\left(T^{\left(n+m\right)n_{1}}x\right)1_{A_{2}}\left(T^{\left(n+m\right)n_{1}}x\right)1_{A_{1}}\left(T^{nn_{1}}x\right)1_{A_{2}}\left(T^{nn_{2}}x\right)d\mu\right)-\mu\left(A_{1}\right)^{2}\mu\left(A_{2}\right)^{2}$}{\Large\par}

{\Large{}$=p-{\displaystyle \lim_{m}}\left(p-{\displaystyle \lim_{n}}\int1_{A_{1}}\left(T^{mn_{1}}x\right)1_{A_{1}}\left(x\right)1_{A_{2}}\left(T^{n(n_{2}-n_{1})}x\right)1_{A_{2}}\left(T^{m(n_{2}-n_{1})}x\right)d\mu\right)-\mu\left(A_{1}\right)^{2}\mu\left(A_{2}\right)^{2}$}{\Large\par}

{\Large{}$=p-{\displaystyle \lim_{m}}\left(\left(\int1_{A_{1}}\left(T^{nn_{1}}x\right)1_{A_{1}}\left(x\right)d\mu\right)\left(\int1_{A_{2}}\left(T^{mn_{n}}x\right)1_{A_{2}}\left(x\right)d\mu\right)\right)-\mu\left(A_{1}\right)^{2}\mu\left(A_{2}\right)^{2}$ }{\Large\par}

{\Large{}$=\left(\int1_{A_{1}}d\mu\right)\left(\int1_{A_{1}}d\mu\right)\left(\int1_{A_{2}}d\mu\right)\left(\int1_{A_{2}}d\mu\right)-\mu\left(A_{1}\right)^{2}\mu\left(A_{2}\right)^{2}$}{\Large\par}

{\Large{}$=0$}{\Large\par}

{\Large{}By Lemma the above lemma, we have $p-\lim{}_{n}a_{n}=0$
in the weak topology. This proves the theorem.}{\Large\par}
\end{proof}
{\Large{}From the discussion of section 2, we know that quasi central
sets and $C$ sets are essential $\mathcal{PS}$ sets and essential
$C$ sets respectively . So it is reasonable and practical to define
quasi central$^{\star}$mixing and $C^{\star}$ mixing.}{\Large\par}

{\Large{}Following is the definition of quasi central$^{\star}$ mixing:}{\Large\par}
\begin{thm}
{\Large{}The measure preserving dynamical system$(X,B,\mu,T)$ is
quasi central$^{\star}$ mixing if for any $A,B\in\mathcal{B}$ with
$\mu\left(A\right)\mu\left(B\right)>0$ and any $\varepsilon>0$,
the set $\left\{ n\in\mathbb{N}:\mid\mu\left(A\cap T^{-n}B\right)-\mu\left(A\right)\mu\left(B\right)\mid<\varepsilon\right\} $
is a quasi central$^{\star}$ set.}{\Large\par}
\end{thm}

{\Large{}Another definition:}{\Large\par}
\begin{thm}
{\Large{}The measure preserving dynamical system$(X,B,\mu,T)$ is
$C^{\star}$ mixing if for any $A,B\in\mathcal{B}$ with $\mu\left(A\right)\mu\left(B\right)>0$
and any $\varepsilon>0$, the set 
\[
\left\{ n\in\mathbb{N}:\mid\mu\left(A\cap T^{-n}B\right)-\mu\left(A\right)\mu\left(B\right)\mid<\varepsilon\right\} 
\]
 is a $C^{\star}$ set.}{\Large\par}
\end{thm}

{\Large{}We know that, central set $\Rightarrow$ quasi central set
$\Rightarrow$ $D$ set, so $D^{\star}$ set $\Rightarrow$ quasi
central$^{\star}$ set $\Rightarrow$central$^{\star}$set. From theorem
5.1 and theorem 5.2 we get quasi central$^{\star}$ mixing is nothing
but weak mixing. Further we know that $IP^{\star}$ set $\Rightarrow$
$C^{\star}$ set $\Rightarrow$ central$^{\star}$ set. From theorem
5.3 and theorem 5.1, we get mild mixing $\Rightarrow$$C^{\star}$
mixing $\Rightarrow$ weak mixing. We don't know whether $C^{\star}$
mixing is strictly intermediate between mild mixing and weak mixing.}{\Large\par}

\textbf{\Large{}Acknowledges:}{\Large{} The second author acknowledges
the grant UGC-NET SRF fellowship with id no. 421333 of CSIR-UGC NET
December 2016.}{\Large\par}

\end{document}